\documentclass[11pt,reqno]{amsart}

\usepackage{amscd,amssymb,amsmath,amsthm}
\usepackage[arrow,matrix]{xy}
\usepackage{graphicx}
\usepackage{epstopdf}
\usepackage{color}
\newtheorem{thm}{Theorem}
\newtheorem{defn}{Definition}
\newtheorem{lemma}{Lemma}

\newtheorem{rk}{Remark}

\numberwithin{equation}{section} 

\def\w{\omega}

\def\a{\alpha}
\def\w{\omega}

\def\b{\beta}

\begin{document}

\title[Non-volterra quadratic stochastic operator]{A quasi-strictly non-volterra quadratic stochastic operator}

\author{A.J.M. Hardin, U.A. Rozikov}

 \address{A.\ J.\ M. \ Hardin \\ University of Oklahoma,
660 Parrington Oval, Norman, Oklahoma 73019, USA.} \email
{ahardin@ou.edu}

 \address{U.\ A.\ Rozikov \\ Institute of mathematics,
81, Mirzo Ulug'bek str., 100125, Tashkent, Uzbekistan.} \email
{rozikovu@yandex.ru}

\begin{abstract} We consider a four-parameter family of non-Volterra operators defined
on the two-dimensional simplex and show that, with one exception, each such
operator has a unique fixed point.
Depending on the parameters, we establish the type of this fixed point.
We study the set of limit points for each trajectory and
show that this set can be a single point or can contain a 2-periodic trajectory.
\end{abstract}

\maketitle

\vskip 0.3 truecm


\vskip 0.3 truecm

\section{Introduction}

\emph{Quadratic stochastic operators} (QSOs) frequently arise
in many models of mathematical genetics, namely, in the theory of heredity
(see e.g. \cite{Lyu1} and \cite{GMR} for motivations and
results related to QSOs). Here we shall investigate a family
of QSOs defined on the two-dimensional simplex. Let us give some definitions first.
The $(m-1)$-dimensional simplex is defined by
\begin{equation} \label{simplex}
S^{m-1} =\left\{x = (x_1, x_2,\dots,x_m)\in R^m : \text{ for any } i \ x_i \geq 0, \text { and } \sum_{i=1}^m x_i=1\right\}.
\end{equation}

A QSO is the mapping $V:S^{m-1}\rightarrow S^{m-1}$ with
 \begin{equation} \label{gqso1}
(V x)_k = \sum_{i,j=1}^m P_{ij,k}x_ix_j
\end{equation}
 where
  \begin{equation} \label{gqsokoef} P_{ij,k}\geq 0, \ \ P_{ij,k} = P_{ji,k},
  \ \  \sum_{k=1}^m P_{ij,k}=1  \ \ \text {for all} \  i, j, k.
\end{equation}

For a given $x^{(0)}\in S^{m-1}$ the trajectory
$x^{(n)}$ of $x^{(0)}$ under the action of the QSO \eqref{gqso1} is
defined by $x^{(n+1)} = V (x^{(n)})$,  where $n = 0, 1, 2,\dots$.

One of the main problems in mathematical biology consists in the study of
the asymptotic behavior of these trajectories.

Denote by $\omega (x^{0})$ the set of $\omega-$limiting points of trajectory $x^{(n)}$.
 Since $S^{m-1}$ is a compact set and $\{x^{(n)}\}\subset S^{m-1}$ it follows that $\omega (x^{0})\neq\emptyset$.
It is clear that if $\omega (x^{0})$ consists of a single point, then the trajectory converges and $\omega (x^{0})$
is a fixed point of (\ref{gqso1}). The limit behavior of the trajectories of any QSO on one-dimensional
space was fully studied by Yu.I. Lyubich \cite{Lyu2}. However, the problem is still open even in the two-dimensional simplex.

\begin{defn}\cite{RJ2}\label{dsnvqso} A quadratic stochastic operator 
  is called Volterra if
\begin{equation} \label{eqv4}
       p_{ij,k}=0, \ \ \mbox {for} \ \ k\notin \{i,j\},  \  \ i,j,k=1,...,m,
\end{equation}
and it is called strictly non-Volterra if
\begin{equation} \label{eqv5}
       p_{ij,k}=0, \ \ \mbox {for} \ \ k\in \{i,j\},  \  \ i,j,k=1,...,m
\end{equation}
\end{defn}

In \cite{GJM2, RJ3} a Volterra operator of a bisexual population was investigated.
However, in the non-Volterra case, many
questions remain open and there seems to be no general theory available \cite{BlJaSc, JM,RJ1,RJ2, RZ, RZ1}.
In \cite{RJ2} the conception of strictly non-Volterra QSOs was introduced, and it was proved that
an arbitrary strictly non-Volterra quadratic stochastic operator on the
two-dimensional simplex has a unique fixed point, which is not attracting.

\begin{rk}
 A strictly non-Volterra operator exists only if $m\geq3$. In \cite{RJ2} the case $m=3$ was studied.
 In the present paper we will study the dynamics of the case $m=3$ for a non-Volterra QSO which has a strictly non-Volterra $x_1^\prime$ and $x_2^\prime$.
\end{rk}

In this paper we consider a non-Volterra QSO (which we call quasi-strictly non-Volterra)
defined on the two-dimensional simplex which has the form
\begin{equation} \label{eqv6}
V:\left\{\begin{array}{llll}
x^\prime_1= \a x^2_2 + cx^2_3 + 2x_2x_3,\\[4mm]
x^\prime_2=ax^2_1 + dx^2_3 + 2x_1x_3,\\[4mm]
x^\prime_3=bx^2_1 + \b x^2_2 + ex^2_3 + 2x_1x_2,
\end{array}\right.
\end{equation}
where   $\a, \b, a, b, c, d, e \geq 0$  and
\begin{equation} \label{eqv7}
      \begin{array}{ll} a+b=1, \hspace{2mm}\a+\b=1, \hspace{2mm}c+d+e=1.
\end{array}
\end{equation}

The paper is organized as follows. In Section 2 we study fixed points of (\ref{eqv6}),
where we show that for $e\ne 1$, the QSO (\ref{eqv6}) has a unique fixed point.
In the case $e=1$, we show that the operator has two fixed points. In Section 3 we find conditions on parameters
under which a fixed point is a repelling, attracting, or saddle point. In the last section we  describe the $\omega$-limit set of this non-Volterra QSO on $S^2$.

\section{Fixed point of operator}

\begin{defn} A point $x\in S^{m-1}$ is called a fixed point of a QSO $V$ if $V(x)=x$, i.e. it satisfies
\begin{equation} \label{eqv8}
\left\{\begin{array}{llll}
x_1= \a x^2_2 + cx^2_3 + 2x_2x_3,\\[4mm]
x_2=ax^2_1 + dx^2_3 + 2x_1x_3,\\[4mm]
x_3=bx^2_1 + \b x^2_2 + ex^2_3 + 2x_1x_2.
\end{array}\right.
\end{equation}
\end{defn}

\begin{thm} \label{Th1}
The non-Volterra QSO (\ref{eqv6}) has a unique fixed point $x^*=(x^*_1,x^*_2,x^*_3 )\in S^2$ in all cases, except when $e=1$. In the case $e=1$, there are two fixed points of the system, one of which is $(0,0,1)$.
\end{thm}

\begin{proof}  We shall consider all possible cases on $a$ and $\a$.

1) Let $\a \neq 0, a \neq 0$. Substituting $x_1 = 1 - x_2 - x_3$ into the first equation of (\ref{eqv8}) gives
\begin{equation} \label{eqv9}
x_2 = \frac{-2x_3-1+\sqrt{4(1- \a c) x^2_3 + 4(1 - \a)x_3 + 1 + 4\a}}{2\a} \geq 0,
\end{equation}
where $x_3 \in [0, \frac{\sqrt{1 + 4c} -1}{2c}]$ if $c \neq 0$ and $x_3 \in [0, 1]$  if $c = 0$.
Similarly, the second equation in (\ref{eqv8}) gives
\begin{equation} \label{eqv10}
x_1 = \frac{-2x_3-1+\sqrt{4(1- ad) x^2_3 + 4(1 - a)x_3 + 1 + 4a}}{2a} \geq 0,
\end{equation}
where $x_3 \in [0, \frac{\sqrt{1 + 4d} -1}{2d}]$ if $d \neq 0$ and $x_3 \in [0, 1]$  if $d = 0$.

Now we may substitute (\ref{eqv9}) and (\ref{eqv10}) into $1 = x_1 + x_2 + x_3$, allowing $x = x_3$ and $f(x) = x$. This gives the function
\begin{align*}
f(x) = &\frac{\a\sqrt{4(1-ad)x^2+4(1-a)x+1+4a}}{2(\a+a-\a a)} \\
 &+\frac{a\sqrt{4(1-\a c)x^2+4(1-\a)x+1+4\a}-\a-a-2\a a}{2(\a+a-\a a)}.
\end{align*}

Now we define
\[\begin{array}{llll}
g(x) = 4(1-ad)x^2+4(1-a)x+1+4a \geq 0,\\
h(x) = 4(1-\a c)x^2+4(1-\a)x+1+4\a \geq 0.
\end{array}\]
Thus, we have
\[\begin{array}{llll}
g^\prime(x) = 8(1-ad)x+4(1-a) \geq 0, & h^\prime(x) = 8(1-\a c)x+4(1-\a) \geq 0,\\
g^{\prime\prime}(x) = 8(1-ad) \geq 0,& h^{\prime\prime}(x) = 8(1-\a c)x \geq 0.
\end{array}\]
Then,
\[
f(x) = \frac{\a\sqrt{g(x)}}{2(\a+a-\a a)} +\frac{a\sqrt{h(x)}-\a-a-2\a a}{2(\a+a-\a a)}.
\]
Differentiating $f(x)$ gives
\[
f^\prime(x) =  \frac{\a g^\prime(x)}{4(\a+a-\a a)\sqrt{g(x)}}
+ \frac{ah^\prime(x)}{4(\a+a-\a a)\sqrt{h(x)}} \geq 0,
\]

\begin{align*}
f^{\prime\prime}(x) = & \frac{\a}{4(\a+a-\a a)}
\frac{g(x)g^{\prime\prime}(x) - \frac{1}{2}g^\prime(x)^2}{\sqrt{g(x)^3}} \\
& + \frac{a}{4(\a+a-\a a)}
\frac{h(x)h^{\prime\prime}(x) - \frac{1}{2}h^\prime(x)^2}{\sqrt{h(x)^3}} \geq 0.
\end{align*}

These inequalities follow from (\ref{eqv7}) as well as the fact that $\a + a - \a a >0$. Additionally, substituting the values for $g(x)$, $h(x)$, and their derivatives into the inequalities
$$ \begin{array}{llll}
g(x)g^{\prime\prime}(x) - \frac{1}{2}g^\prime(x)^2 \geq 0, & h(x)h^{\prime\prime}(x) - \frac{1}{2}h^\prime(x)^2 \geq 0,
\end{array} $$
gives
$$ \begin{array}{llll}
(4(1-ad)x^2+4(1-a)x+1+4a)(1-ad) - (2(1-ad)x+(1-a))^2 \geq 0, \\
(4(1-\a c)x^2+4(1-\a)x+1+4\a)(1-\a c) - (2(1-\a c)x+(1-\a))^2 \geq 0.
\end{array} $$
This can be reduced to
$$ \begin{array}{llll}
\a (-6 + a + d + 4 a d) \leq 0, & a(-6 + \a + c + 4 \a c) \leq 0.
\end{array} $$
Thus the inequalities $g(x)g^{\prime\prime}(x) - \frac{1}{2}g^\prime(x)^2 \geq 0$ and $h(x)h^{\prime\prime}(x) - \frac{1}{2}h^\prime(x)^2 \geq 0$ can be demonstrated to be true.

Additionally,
\[
f(0) = \frac{\a(\sqrt{1+4a} - 1 -a)  +a(\sqrt{1+4\a}-1-\a)}{2(\a+a-\a a)} > 0,
\]
\[
f(1) =\frac{\a\sqrt{9-ad} +a\sqrt{9-\a c}-\a-a-2\a a}{2(\a+a-\a a)} \leq 1,
\]
which follows from the above inequalities in addition to the fact that
$$ \begin{array}{llll}
\sqrt{9-ad} \leq 3, & \sqrt{9-\a c} \leq 3.
\end{array} $$
Additionally, both $\sqrt{9-ad}$ and $\sqrt{9-\a c}$ can only be simultaneously equal to $3$ when $c=d=0$ (and thus when $e=1$). This means that $f(1) = 1$ when $e=1$ and $f(1) < 1$ otherwise.

Thus the function is increasing and convex in $[0,1]$. Therefore, for $e\in[0,1)$ the
system has a unique fixed point, since $f(x)$ will only intersect the line $x$ at one point in the domain $[0,1]$. When $e=1$ the system has a fixed point $(0,0,1)$.
It can also be directly shown that when $e=1$, $f(\frac{9}{10}) < \frac{9}{10}$.
This demonstrates that the function $f(x)$ must cross the line $x$ prior to reaching the value $f(1) = 1$. This proves that there are two fixed points when $e=1$.

2) Let $\a \neq 0, a = 0$ can be handled similarly.
Substituting $x_1 = 1 - x_2 - x_3$ into the first equation of (\ref{eqv8}) gives
\begin{equation} \label{eqv9again}
x_2 = \frac{-2x_3-1+\sqrt{4(1- \a c) x^2_3 + 4(1 - \a)x_3 + 1 + 4\a}}{2\a} \geq 0,
\end{equation}
where $x_3 \in [0, \frac{\sqrt{1 + 4c} -1}{2c}]$ if $c \neq 0$ and $x_3 \in [0, 1]$  if $c = 0$. The second equation in (\ref{eqv8}) gives
\begin{equation} \label{eqv10again}
x_1 = \frac{1-x_3-d x_3^2}{2x_3+1} \geq 0,
\end{equation}
where $x_3 \in [0,\frac{\sqrt{1+4d}-1}{2d}]$
if $d \neq 0$ and $x_3 \in [0, 1]$  if $d = 0$. This restriction ensures that $x_1$ is positive.

Directly substituting (\ref{eqv9again}) and (\ref{eqv10again})
into $1 = x_1 + x_2 + x_3$, allowing $x = x_3$, gives
$$
1 = x +  \frac{1-x-d x^2}{2x+1} + \frac{-2x-1+\sqrt{4(1- \a c) x^2 + 4(1 - \a)x + 1 + 4\a}}{2\a}.
$$
Solving for the $x$s present in the first two terms of the above equation (but not any of the $x$s in the (\ref{eqv9again}) term), and allowing $F_\pm(x) = x$, gives the functions
\begin{align*}
F_\pm(x) = \frac{1+2\a+2x-\sqrt{h(x)}\pm \sqrt{q(x)}}{2\a(2-d)},
\end{align*}
where
\[\begin{array}{llll}
h(x) = 4(1-\a c)x^2+4(1-\a)x+1+4\a \geq 0,\\
q(x) = 2\a(2-d)(1-\sqrt{h(x)}+2x)+(1-\sqrt{h(x)}+2x + 2\a)^2.
\end{array}\]

Two examples of the function $F_\pm$ graphed against the line $x$ are given below.
\begin{figure}[h]
\begin{center}
    \includegraphics[width=11cm]{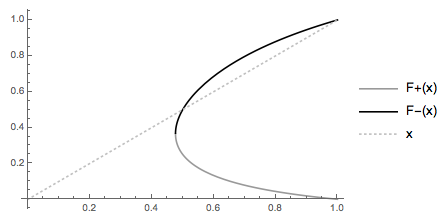}
    \caption{$\a = .0001, e = 1$.}
    \includegraphics[width=11cm]{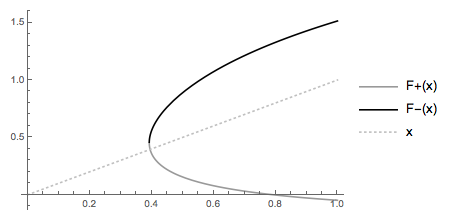}
    \caption{$\a = .0001, e = 0.18$.}
 \end{center}
    \label{fig:coffee}
\end{figure}

\begin{lemma}
$F_\pm(x) = x$ at a unique point when  $e \in [0,1)$ and at two points --- one of which is $F_+(1)=1$ --- when $e=1$.
\end{lemma}
\begin {proof}
We note that when $q(x) = 0$,
\begin{align*}
x_q = &\frac{1}{2c}[-5 + d + \sqrt{12-8d+d^2} \\
&+  (37+2d^2-10 \sqrt{12-8d+d^2} + 2d(-9+\sqrt{12-8d+d^2}) \\
&- 2c(2-d- \sqrt{12-8d+d^2} + \a(14+d^2-4 \sqrt{12-8d+d^2}+d(-8+ \sqrt{12-8d+d^2}))))^{\frac{1}{2}}].
\end{align*}

The curve $F_\pm$ is imaginary when $x < x_q$. When $x = x_q$, $F_+(x) = F_-(x)$, and when  $x > x_q$, $F_+(x) > F_-(x)$.

Additionally, at $x=1$ we have
$$
F_+(1) = \frac{3+2\a-\sqrt{h(1)} + \sqrt{q(1)}}{2\a(2-d)} \geq 1,
$$
where
\[\begin{array}{llll}
h(1) = 9-4\a c \geq 0,\\
q(1) = 2\a(2-d)(3-\sqrt{h(1)})+(3-\sqrt{h(1)} + 2\a)^2.
\end{array}\]
It can be readily shown that in the case that $e=1$, $F_+(1) = 1$. When $e \in [0,1)$, we would like to show that $F_+(1) > 1$. This can be reduced to showing that $\sqrt{q(1)}+2\a d + 3 > \sqrt{h(1)}+2\a$. It can be easily shown that $\sqrt{q(1)} > 2\a$.
Additionally, when $d=0$, $3 > \sqrt{h(1)}$. When $d \neq 0$,  $3 \geq \sqrt{h(1)}$, but $2\a d > 0$. The culmination of the above facts proves the inequality $\sqrt{q(1)}+2\a d + 3 > \sqrt{h(1)}+2\a$ to be always true.

Thus, $F_+(1) > 1$ when $e<1$. It can be similarly shown that $mF_-(1) < 1$. Analysis of the derivations of $F_\pm$ shows that the following inequalities are true.
$$
F_+'(x) = \frac{1}{2\a(2-d)}\left(2-\frac{h'(x)}{2\sqrt{h(x)}} + \frac{q'(x)}{2\sqrt{q(x)}}\right)  \geq 0,
$$
$$
F_-'(x) =  \frac{1}{2\a(2-d)}\left(2-\frac{h'(x)}{2\sqrt{h(x)}} - \frac{q'(x)}{2\sqrt{q(x)}}\right)  \leq 0,
$$
$$
F_+''(x) =  \frac{1}{2\a(2-d)}\left(-\frac{h(x)h''(x)-\frac{1}{2}(h'(x))^2}{\sqrt{(h(x))^3}} +
\frac{q(x)q''(x)-\frac{1}{2}(q'(x))^2}{\sqrt{(q(x))^3}}\right)  \leq 0,
$$
for all $x > x_q$.

The above demonstrates that $F_\pm(x)$ must intersect the line $x$ at least once on $[0,1)$. Because $F_+(x)$ is increasing and concave, and $F_-(x)$ is decreasing, $F_\pm(x)$ and $x$ will only intersect once on $[0,1)$. Additionally, $F_+(1)=1$ only when $e=1$. Thus the lemma is proven.
\end{proof}

Therefore, a direct extension of the above lemma shows that when $e=1$, the system (\ref{eqv6}) has two fixed points$-$one of which is $(0,0,1)$. And when $e<1$, it has a unique fixed point.

3) The case $\a = 0, a \neq 0$ can be handled analogously to the previous case.

4)  Let $\a = a = 0$. The system is therefore
\begin{equation}
\left\{\begin{array}{llll}  \label{eqv12}
x_1=cx^2_3 + 2x_2x_3\\[4mm]
x_2=dx^2_3 + 2x_1x_3\\[4mm]
x_3= (x_1 + x_2)^2 +  ex^2_3 = (1 - x_3)^2 +  ex^2_3
\end{array}\right.
\end{equation}

The proof of the fourth case will follow directly from the following lemma.
\begin{lemma} \label{smplFxd}
If $e \in [0,1)$, the system has a unique fixed point, $(x^*_1, x^*_2, x^*_3)$, where
\begin{align*}
&x^*_1 = \frac{7+24d-2e-19de-5e^2+2e^2d + (3-11d+4de+3d^2)\sqrt{5-4e}}{2 (1 + e) (-13 + 6  \sqrt{5 - 4 e} + 6 e + e^2)},\\
 &x^*_2 =  1 -  \frac{7+24d-2e-19de-5e^2+2e^2d + (3-11d+4de+3d^2)\sqrt{5-4e}}{2 (1 + e) (-13 + 6  \sqrt{5 - 4 e} +   6 e + e^2)} - \frac{3 - \sqrt{5-4e}}{2(1+e)}, \\
   & x^*_3 = \frac{3 - \sqrt{5-4e}}{2(1+e)}.
 \end{align*}
   However, when $e = 1$ the system has two fixed points: $(0, 0, 1)$ and $(\frac{1}{4}, \frac{1}{4}, \frac{1}{2})$.
\end{lemma}

\begin{proof} Examine each possible case.
\begin {itemize}
\item Let $e \in [0,1)$. Solving the third equation of (\ref{eqv12}) for $x_3$ gives $x^*_3 = \frac{3 \pm \sqrt{T}}{2(1+e)} > 0$ where $T = 5-4 e > 0$.

Assume for the purpose of contradiction that $\frac{3 + \sqrt{T}}{2(1+e)} \leq 1$. This can be reduced to $ \sqrt{T} \leq 2 e-1$. If $e \leq \frac{1}{2}$, then $2 e-1 \leq 0$ and $\sqrt{T} \geq 0$, so $\frac{3 + \sqrt{T}}{2(1+e)} \nleq 1$. Additionally, if $e > \frac{1}{2}$ then $0 <  \sqrt{T} \leq 2 e-1$ which can be reduced to $1 \leq e$ which is not true for $e \in [0,1)$. Thus $\frac{3 + \sqrt{T}}{2(1+e)} \nleq 1$ for any $e \in [0,1)$ and is not a fixed point of $x_3^\prime = (1-x_3)^2 + ex_3^2$.

It can be proved that $\frac{3 - \sqrt{T}}{2(1+e)} \leq 1$ because it can be reduced to $ \sqrt{T} \geq 1-2 e$. As it was shown previously that , $ \sqrt{T} \leq 1-2 e$ is false for all $e \in [0,1)$, it must be that $ \sqrt{T} > 1-2 e$.Therefore,
$$
x^*_3 = \frac{3 - \sqrt{5-4e}}{2(1+e)},
$$
is a unique fixed point of the system.

Substituting $x^*_3$ into the first two equations of (\ref{eqv12}) gives $x_1 = c{x^*_3}^2 + 2 x_2 x^*_3$ and $x_2 = d{x^*_3}^2 + 2 x_1 x^*_3$. Substituting this value of $x_2$ into $x_1$ and reducing gives
$$
x^*_1 = \frac{(3x^*_3-1)(c + 2dx^*_3)}{5 + e-12x^*_3},
$$
which can be written as
$$
x^*_1 = \frac{7+25d-2e-19de-5e^2+2e^2d + (3-11d+4de+3d^2)\sqrt{5-4e}}{2 (1 + e) (-13 + 6  \sqrt{5 - 4 e} +
   6 e + e^2)}.
$$
Additionally, we know that $x^*_2 = 1 - x^*_1 - x^*_3$ which yields
$$
x^*_2 = 1 -  \frac{7+25d-2e-19de-5e^2+2e^2d + (3-11d+4de+3d^2)\sqrt{5-4e}}{2 (1 + e) (-13 + 6  \sqrt{5 - 4 e} +
   6 e + e^2)} - \frac{3 - \sqrt{5-4e}}{2(1+e)}.
$$
Thus, $(x^*_1, x^*_2, x^*_3)$ is a unique fixed point of the system.

\item Let $e = 1$. The system is therefore
\begin{equation}
\left\{\begin{array}{llll}  \label{dbleFxd}
x_1= 2x_2x_3,\\[4mm]
x_2= 2x_1x_3,\\[4mm]
x_3= (1 - x_3)^2 +  x^2_3.
\end{array}\right.
\end{equation}
Additionally, we know that $x^*_3 = \frac{3\pm\sqrt{(5 - 4e)}}{2(1+e)}$; therefore, $x^*_3 = \frac{1}{2}$ or $1$.
When $x^*_3 = \frac{1}{2}$, it follows from (\ref{dbleFxd}) that $x_1 = x_2$, and it follows from $x_1 + x_2 + x_3 = 1$ that $x_1 + x_2 = \frac{1}{2}$. Thus, $x^*_1 = x^*_2 = \frac{1}{4}$ and the point $(\frac{1}{4}, \frac{1}{4}, \frac{1}{2}) \in S^2$ is a fixed point of the system. When $x^*_3 = 1$, it follows from $x_1 + x_2 + x_3 = 1$ that $x^*_1 = x^*_2 = 0$. Thus, the point $(0, 0, 1) \in S^2$ is a second fixed point of the system.
\end{itemize}
\end{proof}

By the above cases, all possible values for the system are considered and the theorem is proved.
\end{proof}

\section{ The type of the fixed point}

\begin{defn}\cite{De}. A fixed point $x^*$ of the operator  $V$ is called hyperbolic if its
Jacobian $J$  at $x^*$ has no eigenvalues on the unit circle.
\end{defn}

\begin{defn} \cite{De}. A hyperbolic fixed point $x^*$ is called:
\begin{itemize}
  \item [i)]  attracting  if all the eigenvalues of the Jacobian $J(x^*)$ are less than 1 in absolute value;
  \item [ii)] repelling   if all the eigenvalues of the Jacobian $J(x^*)$ are greater than 1 in absolute value;
 \item [iii)] a saddle    otherwise.
 \end{itemize}
\end{defn}

To find the type of a fixed point we use $x_3=1-x_1-x_2$ to rewrite QSO (\ref{eqv6}) as follows:
$$
V:\left\{\begin{array}{llll}
x^\prime_1= & c-2c x_1 + c x_1^2 + 2(c-1)x_1 x_2 + 2(1-c)x_2 + (\a + c -2)x_2^2,\\[3mm]
x^\prime_2= & d-2d x_2 + d x_2^2 + 2(d-1)x_1 x_2 + 2(1-d)x_1 + (a + d -2)x_1^2,
\end{array}\right.
$$
where $(x_1,x_2)\in\{(x,y):x,y\geq0, \ \ 0\leq x+y\leq1\}$ and  $x_1,x_2,$ are the first
two coordinates of a point lying in the two-dimensional simplex.

The Jacobian, $J(x^*)$, has the representation
\begin{equation} \label{eqv14}
\left (\begin{array}{lll}
-2 c (1 - x^*_1 - x^*_2) - 2 x^*_2 &
 2(\a-1)x^*_2 + 2(1-c) (1 - x^*_1 - x^*_2) \\
 2(a-1)x^*_1 + 2 (1-d)(1 - x^*_1 - x^*_2) &
- 2 d (1 - x^*_1 - x^*_2)  -2 x^*_1
\end{array}\right).
\end{equation}

The Jacobian (\ref{eqv14}) has the eigenvalues $\lambda_{1,2} = ex^*_3 -1 \pm  \sqrt{D(a, \a, c, e)}$, where
\begin{multline} \label{D}
D \equiv D(a, \a, c, e) =  (e x^*_3-1)^2 +4e{x_3^*}^2\\
+ 4[(b\b-1)x^*_1 x^*_2 +(a(1-c)-1)x^*_1 x^*_3+(\a(1-d)-1)x^*_2 x^*_3].
\end{multline}

The classification of these eigenvalues is as follows:
\begin{equation} \label{eigenClass}
\left\{\begin{array}{llll}
\text{If } D < -1+(1- ex^*_3)^2,  & \text{the fixed point is repelling;}\\[4mm]
\text{If }D = -1+(1- ex^*_3)^2,  & \text{the fixed point is nonhyperbolic;}\\[4mm]
\text{If } -1+(1- ex^*_3)^2 < D < 0,  & \text{the fixed point is attracting;}\\[4mm]
\text{If }D = 0 \text{ and } ex^*_3 = 0, & \text{the fixed point is nonhyperbolic;}\\[4mm]
\text{If }D = 0 \text{ and } ex^*_3 > 0, & \text{the fixed point is attracting;}\\[4mm]
\text{If }0 < D < e^2{x^*}^2_3, & \text{the fixed point is attracting;}\\[4mm]
\text{If }D = e^2{x^*}^2_3,  & \text{the fixed point is nonhyperbolic;}\\[4mm]
\text{If }e^2{x^*}^2_3  < D < (2 - ex^*_3)^2,  & \text{the fixed point is a saddle point;}\\[4mm]
\text{If }D = (2 - ex^*_3)^2,  & \text{the fixed point is nonhyperbolic;}\\[4mm]
\text{If }(2- ex^*_3)^2 < D, & \text{the fixed point is repelling.}
\end{array}\right.
\end{equation}

In  \cite{RJ2} it was proven that strictly non-Volterra QSOs with $m=3$ have a
unique fixed point and that the type of the hyperbolic fixed point can never be attracting. However, in the system (\ref{eqv6}),
the introduction of the parameter $e$ has caused an attracting fixed point to become possible,
as evidenced by the following example.

{\bf Example.} When $c=\b=\frac{5}{8}$, $d=b=0$, $\a=e=\frac{3}{8}$, and $a=1$ the system (\ref{eqv6}) can be written
\begin{equation} \label{eqvE}
V:\left\{\begin{array}{llll}
x^\prime_1= \frac{3}{8} x^2_2 + \frac{5}{8}x^2_3 + 2x_2x_3,\\[4mm]
x^\prime_2=x^2_1 + 2x_1x_3,\\[4mm]
x^\prime_3=\frac{5}{8}x^2_2 + \frac{3}{8}x^2_3 + 2x_1x_2.
\end{array}\right.
\end{equation}
The fixed point of this system is $(\frac{1}{3},\frac{1}{3},\frac{1}{3})$. Substituting these values into the eigenvalues of the Jacobian gives $\lambda_{1,2} = -\frac{7}{8} \pm \frac{1}{8}\sqrt{\frac{13}{3}}\hspace{1mm}i.$ Therefore, $|\lambda_{1,2}| = \sqrt{\frac{5}{6}}<1,$ and $(\frac{1}{3},\frac{1}{3},\frac{1}{3})$ is attracting.

In the case where $\a=a=0$ (i.e. system (\ref{eqv12})), the eigenvalues of the Jacobian can be written
$$
\lambda_{1,2} = e x_3^* - 1 \pm \sqrt{1 + 2(e-2)x_3^* + (4+e^2){x_3^*}^2},
$$
where $x_3^* = \frac{3-\sqrt{5-4e}}{2(1+3)}$. It can be proven that $|\lambda_{1}| < 1$ and  $|\lambda_{2}| > 1$ for all $e$.
The inequality $|\lambda_{1}| < 1$ can be proven from the facts that
 \begin{align*}
&0 \leq 1 + 2(e-2)x_3^* + (4+e^2){x_3^*}^2,\\
 &0 <  e x_3^* + \sqrt{1 + 2(e-2)x_3^* + (4+e^2){x_3^*}^2 },\\
 &{x^*_3}^2 + e x^*_3 < \frac{3}{4} + x^*_3.
 \end{align*}
 The last inequality,  ${x^*_3}^2 + e x^*_3 < \frac{3}{4} + x^*_3$, follows from the fact that ${x^*_3}^2 \leq x^*_3$ and from substituting the value of $x^*_3$ into the inequality $e x^*_3 < \frac{3}{4}$, which gives $3e < 3 + 2e\sqrt{5-4e}$. Additionally, $|\lambda_{2}| > 1$ can be proven from the fact that $e x_3^* - 1 - \sqrt{1 + 2(e-2)x_3^* + (4+e^2){x_3^*}^2 } < -1$. This can be reduced to the quadratic $1 + (2e-4)x^*_3+4{x^*_3}^2$, which is always positive.
This means that the fixed point $x_3^*$ is a saddle point for all $e \in [0,1)$. This can be seen clearly in the graph below.

\begin{figure}[h!]
  \includegraphics[width=11cm]{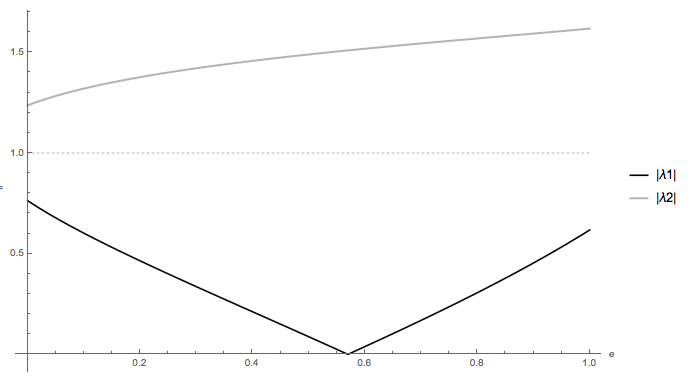}
  \caption{Effect of $e$ on eigenvalues.}
  \label{fig:boat1}
\end{figure}

In the case that $a=\a=0$ and $e=1$, the fixed point $(\frac{1}{4}, \frac{1}{4}, \frac{1}{2})$ is a saddle point and $(0,0,1)$ is a repeller.

Additionally, for all cases where $e=1$ and $(0,0,1)$ is a second fixed point of the system, it can be seen that $\lambda_{1,2} = \pm 2$. Thus the second fixed point that occurs when $e=1$ is always a repeller.

\begin{rk}
A non-Volterra QSO with $m=3$  (\ref{eqv6}) also has a unique fixed point in all cases except $e=1$ for which there are two fixed points, one of which$-(0,0,1)-$is always a repeller. Additionally, the fixed point of a non-Volterra QSO may be attracting.
\end{rk}

\section{The $\omega-$limit set}

In this section we shall describe the $\omega$-limit set of trajectories under certain parameter restrictions.
Let $x^0=(x^0_1, x^0_2, x^0_3) \in S^2$ be the initial point and let $\{x^{(n)}, n = 0,1,2,\dots\}$ be
the trajectory of $x^0$ under the action of the operator (\ref{eqv6}); that is,
$$
 x^{(n)}=(x^{(n)}_1, x^{(n)}_2, x^{(n)}_3)=V(x^{(n-1)}), \hspace{6mm} n = 1,2,\dots  \hspace{4mm}  x^{(0)} = x^0.
$$

For simplicity we shall examine the case in which $a = \a = 0$; therefore the operator can be written as (\ref{eqv12}), i.e.,
\begin{equation}
\left\{\begin{array}{llll}  \label{eqv15}
x'_1=cx^2_3 + 2x_2x_3,\\[3mm]
x'_2=dx^2_3 + 2x_1x_3,\\[3mm]
x'_3= (1 - x_3)^2 +  ex^2_3,
\end{array}\right.
\end{equation}
which demonstrates that the trajectory of the third coordinate $\{x^{(n)}_3\}$ is
defined by the dynamical system of
 $$\varphi(x) = (1-x)^2 + ex^2.$$

 \subsection{ Case $e=1$} In this case operator has the form (denoted by $V_1$)
\begin{equation}
V_1: \ \ \left\{\begin{array}{llll}  \label{eqv151}
x'_1=2x_2x_3,\\[3mm]
x'_2=2x_1x_3,\\[3mm]
x'_3= (1 - x_3)^2 + x^2_3.
\end{array}\right.
\end{equation}
This operator has been studied in \cite{H}:
It is easy to see that
$${\rm Fix}(V_1)=\{(1/4,1/4,1/2), (0,0,1)\}.$$
For any $(x,y,1/2)$ with $x+y=1/2$, we have
$$V^2_1(x,y,1/2)=V_1(y,x,1/2)=(x,y,1/2).$$
In fact the set of all 2-periodic points is
$${\rm Per}_2(V_1)=\{(x,y,1/2)\in S^2: x+y=1/2\}.$$
Moreover it is easy to see that $V_1(x,y,0)=(0,0,1)$, for any $(x,y,0)\in S^2$.

For $\theta\in [0, \infty)$, denote
$$M_\theta=\left\{\begin{array}{ll}
\{(x,y,z)\in S^2: xy=0\}, \ \ \mbox{if} \ \ \theta=0,\\[2mm]
\{(x,y,z)\in S^2: x=\theta y \ \ or \ \ x={1\over \theta}y\}, \ \ \mbox{if} \ \ \theta\in (0, \infty).
\end{array}\right.$$

\begin{thm} \label{eqvA}\cite{H} If $e=1$, then
\begin{itemize}
\item[1.] For any initial point $(x_1^0, x_2^0, x_3^0)$, with $x_3^0=0$ or $x_3^0=1$ we have
$$
 \lim_{n\rightarrow \infty} V_1^n (x_1^0, x_2^0, x_3^0) =(0,0,1).$$
 \item[2.] For any initial point $x^0=(x_1^0, x_2^0, x_3^0)$, with $x_3^0 \ne 0$ and  $x_3^0 \ne 1$, there exists $\theta\in [0, +\infty)$, such that
$x^0\in M_\theta$. Moreover,
 $$
 \lim_{n\rightarrow \infty} V_1^n (x_1^0, x_2^0, x_3^0) =
\left\{\begin{array}{ll}
\left(\frac{\theta}{2(\theta+1)}, \frac{1}{2(\theta+1)}, \frac{1}{2}\right), \ \ n=2k, k=1,2,3,...\\[2mm]
\left(\frac{1}{2(\theta+1)}, \frac{\theta}{2(\theta+1)}, \frac{1}{2}\right), \ \ n=2k+1, k=1,2,3,...
\end{array}\right.
$$
\end{itemize}
\end{thm}

\subsection{ Case $e\in[0,1)$}
Now we consider the operator (\ref{eqv15}) for all $e\ne 1$.
By the above given results we know that this operator has a unique fixed point:
$${\rm Fix}(V)=\{x^*=(x_1^*, x_2^*, x_3^*)\},$$ which is never attractive.

Let us describe periodic points of the operator. By (\ref{eqv15}) the sequence $x^{(n+1)}=V^n(x^0)$ has the form
\begin{equation}
\left\{\begin{array}{llll}  \label{en}
x^{(n+1)}_1=c(x^{(n)}_3)^2 + 2x^{(n)}_2x^{(n)}_3,\\[3mm]
x^{(n+1)}_2=d(x^{(n)}_3)^2 + 2x^{(n)}_1x^{(n)}_3,\\[3mm]
x^{(n+1)}_3= \varphi(x^{(n)}_3).
\end{array}\right.
\end{equation}

\begin{lemma}\label{aa} If ${1/4}\leq e<1$ then the operator (\ref{eqv15}) does not have any $n$-periodic point, $n\geq 2$,
 different from the fixed point $x^*$.
\end{lemma}
\begin{proof} First we give analysis of the equation $\varphi(\varphi(x))=x$, (for existence of 2-periodic points) which
has solutions $x^*_3$, $\overline{x}_3 = \frac{1+\sqrt{1-4 e}}{2 (1+e)}$,
and $\overline{\overline{x}}_3 = \frac{1-\sqrt{1-4 e}}{2 (1+e)}$,
where $\varphi(\overline{x}_3) = \overline{\overline{x}}_3 $ and vice-versa. These numbers exist if and only iff $e<1/4$, and when $e=1/4$
the three numbers coincide.
Thus for ${1/4}\leq e<1$ there are no 2-periodic points of $\varphi$. By Sharkovskii's theorem \cite{De} we have that
$\varphi^n(x)=x$ does not have solution $x\ne x_3^*$ for all $n\geq 2$. Thus, $x_3^{(n)}=x_3$ has unique solution
$x_3=x_3^*$ for any $n\geq 1$. Using this fact we reduce the equation $V^n(x)=x$ (with $x=(x_1,x_2,x_3)$)
 of $n$-periodic points to the
equation $L^n(\hat x)=\hat x$ (with $\hat x=(x_1,x_2)$), where $L$ is the linear operator given by
$$L: \ \ \left\{\begin{array}{llll}  \label{el}
x'=c(x^{*}_3)^2 + 2x^{*}_3y,\\[3mm]
y'=d(x^{*}_3)^2 + 2x^{*}_3x.
\end{array}\right.$$
It is then easy to see that this linear operator has a unique
fixed point $(x_1^*, x_2^*)$ (the first two coordinates of the fixed point $x^*$).
This fixed point is attractive and therefore by the known theorem of
liner dynamical systems (see Chapter 3 of \cite{Gal}), we see that all trajectories of the linear operator tend to the fixed point.
Therefore this linear operator has no periodic points except $x^*$.
\end{proof}
\begin{lemma}\label{ab} If $0\leq e<{1/4}$, then the operator (\ref{eqv15}) has $2$-periodic points, $(\overline{x}_1,\overline{x}_3,\overline{x}_3)$ and $(\overline{\overline{x}}_1,\overline{\overline{x}}_2,\overline{\overline{x}}_3)$$-$different from the fixed point $x^*$$-$which are described explicitly below. Moreover the operator
does not have any $n$-periodic point for all $n\geq 3$.
\end{lemma}
The following proof will rely on the concept of  topological conjugacy.

\begin{defn}
Let $f:A\to A$ and $g: B \to B$ be two maps.
$f$ and $g$ are called topologically conjugate if there exists a homeomorphism $h: A \to B$ such
that, $h \circ f = g \circ h$.
\end{defn}

Additionally, it is known (as shown in \cite{De}) that mappings which are topologically conjugate are completely equivalent in
terms of their dynamics.  In particular, $h$ gives a one-to-one correspondence between periodic points of $f$ and $g$.
\begin{proof}
 It can be seen from $|\varphi^{'}(x_3^*)| = |1-\sqrt{5-4e}|$ that if $e \in [0,\frac{1}{4})$, then
 $x^*_3$ is a repelling fixed point of $\varphi(x)$. As mentioned in the proof of the previous lemma, if
  $e < \frac{1}{4}$ then the function $\varphi(\varphi(x))$ has fixed points
  $x^*_3$, $\overline{x}_3 = \frac{1+\sqrt{1-4 e}}{2 (1+e)}$, and
  $\overline{\overline{x}}_3 = \frac{1-\sqrt{1-4 e}}{2 (1+e)}$, i.e., $\varphi^2(\overline{x}_3) =\overline{x}_3$, and $\varphi^2(\overline{\overline{x}}_3)=\overline{\overline{x}}_3$.
By substituting $x_3= \overline{x}_3,  \overline{\overline{x}}_3$ in the first and second equations of the system
$V^2(x)=x$ and solving it with respect to $x_1$ and $x_2$, we get
\begin{equation}\label{pp}\begin{array}{llll}
\overline{x}_1 = \frac{2de + (1-e-2e^2)c + (2de-c(1+e))\sqrt{1-4e}}{2(e-1)^2(e+1)},\\[3mm]
\overline{\overline{x}}_1 = \frac{2d + (1-e-2e^2)c + (-2d+c(e+1))\sqrt{1-4e}}{2(e-1)^2(e+1)},\\[3mm]
\overline{x}_2 = \frac{2ce + (1-e-2e^2)d + (2ce-d(1+e))\sqrt{1-4e}}{2(e-1)^2(e+1)},\\[3mm]
\overline{\overline{x}}_2 = \frac{2c + (1-e-2e^2)d + (-2c+d(e+1))\sqrt{1-4e}}{2(e-1)^2(e+1)}.
\end{array}\end{equation}
Now we show that the operator has no $n$-periodic point if $n\geq 3$.
It is easy to see that for each solution $x=\tilde x_3$
of $\varphi^n(x)=x$, one gets a unique $(\tilde x_1, \tilde x_2)$ from $V^n(x)=x$. Therefore
the number of periodic points of $V$ is equal to the periodic points of $\varphi$.
Now we show that $\varphi$ does not have $n$-periodic points for any $n\geq 3$.
Taking $h(x)=ax+b$ one can see that our function $\varphi$ is topologically conjugate to the logistic map
$\xi(x)=\mu x(1-x)$ with $\mu=1+\sqrt{5-4e}$. For $e\in [0,{1\over 4})$ we have $\mu\in (3, 1+\sqrt{5}]$.
For the logistic map the following is known (see \cite{Shar}):

If $\mu$ between 3 and $1+\sqrt {5}\approx 3.236$, then $\xi$ has one 2-periodic orbit
and all trajectories (except when started at the fixed point) will approach this 2-periodic orbit.

From this fact, by the conjugacy argument, it follows that $\varphi$, and thus $V$, do not have $n$-periodic point for any $n\geq 3.$
\end{proof}
\begin{rk} The conjugacy argument mentioned in the proof of Lemma \ref{ab} can be also used to give an alternative proof of Lemma \ref{aa}.
In this case $\mu=1+\sqrt{5-4e}\in (2, 3] $ and (see \cite{Shar}) the function  $\xi$ has no periodic points (except fixed points).
All trajectories will converge to the non-zero fixed point.
\end{rk}
\begin{thm} \label{eqvC}
Let $e\in [0,1)$.
\begin{itemize}
\item[1.] If $e<\frac{1}{4}$ then there exists an open set $\mathcal U\subset S^2$ such that $\overline{x}, \overline{\overline{x}}\in \mathcal U$
  and for any $x^0=(x_1^0, x_2^0, x_3^0) \in \mathcal U$ we have
\begin{equation}\label{ak}
 \lim_{n\rightarrow \infty} V^n(x^0)  =
\left\{\begin{array}{llll}
\overline{x}, \ \ \mbox{if} \ \  x_3^0 \neq x^*_3 \text{ and } n=2k,\\[4mm]
x^*, \ \ \mbox{if} \ \  x_3^0 = x^*_3,\\[4mm]
\overline{\overline{x}}, \ \ \mbox{if} \ \    x_3^0 \neq x^*_3 \text{ and } n=2k+1,
\end{array}\right.
\end{equation}
where $x^*=(x^*_1,x^*_2,x^*_3)$ is fixed point and  $\overline{x}=(\overline{x}_1,\overline{x}_3,\overline{x}_3)$, $\overline{\overline{x}}=(\overline{\overline{x}}_1,\overline{\overline{x}}_2,\overline{\overline{x}}_3)$ are periodic points described above.

\item[2.] If $e\geq \frac{1}{4}$ then there exists an open set $U\subset S^2$ such that $x^*\in U$ and
for any $x^0=(x_1^0, x_2^0, x_3^0) \in  U$ we have
 $$\lim_{n\rightarrow \infty} V^n (x_1^0, x_2^0, x_3^0) = (x^*_1, x^*_2, x^*_3).$$
\end{itemize}
\end{thm}
\begin{proof} 1)
For $e \in [0,\frac{1}{4})$, it can be seen from $|\varphi'(x_3^*)| = |1-\sqrt{5-4e}|$
that $x^*_3$ is a repelling fixed point of $\varphi(x)$. Additionally,
when $e < \frac{1}{4}$ the fixed points $\overline{x}_3$ and  $\overline{\overline{x}}_3$
of function $g(x)=\varphi(\varphi(x))$ are attracting, which follow
from $|g'(\overline{x}_3)|<1$ and $|g'(\overline{\overline{x}}_3)|<1$.
Define the operator  $W:[0,1]^2\to [0,1]^2$ by the first and the last coordinate of the operator $V$:
 \begin{equation}\label{ej}
 W: \left\{\begin{array}{ll}
x_1'=cx_3^2+2x_3(1-x_1-x_3),\\[3mm]
x_3'=\varphi(x_3).
\end{array}\right.
\end{equation}
Now using the Jacobian of the operator $W(W(x))$ one can see that the
2-periodic orbit $\{\overline{x}, \overline{\overline{x}}\}$ is a unique attracting
orbit, and the fixed point $x^*$ is a saddle point of $V$. The operator $V$ has the following invariant sets:
$$\gamma=\{(x_1,x_2,x_3)\in S^2: x_3=x_3^*\},$$
$$\Gamma=\{(x_1,x_2,x_3)\in S^2: x_3=\overline{x}_3
\ \ \mbox{or} \ \ \overline{\overline{x}}_3\}.$$
Note that if $x^0\in \gamma$ then $\lim_{n\to\infty}{x_3^{(n)}} = x^*_3.$
If $x^0\in \Gamma$ then $\lim_{n\to\infty}{x_3^{(n)}} = \overline{x}_3$
when $n$ is even and $\lim_{n\to\infty}{x_3^{(n)}} = \overline{\overline{x}}_3$ when $n$ is odd, in these cases
the trajectory for $x_1$ can be written as
$$x_1^{(n+2)} = c \overline{\overline{x}}_2^2 + 2(d \overline{x}_3^2 +
2 x_1^{(n)} \overline{x}_3)\overline{\overline{x}}_3\ \ \mbox{for even} \ \ n,$$
and
$$x_1^{(n+2)} = c \overline{x}_2^2 + 2(d \overline{\overline{x}}_3^2 +
2 x_1^{(n)} \overline{\overline{x}}_3)\overline{x}_3 \ \ \mbox{for odd} \ \ n.$$
The existence of the limit (\ref{ak}) follows from general theory of dynamical systems (see \cite{De})
and the uniqueness of the attracting 2-periodic points.

2) Next we shall consider when $e = \frac{1}{4}$. It can be seen from $|\varphi'(x_3^*)| = |1-\sqrt{5-4e}|$ that when $e = \frac{1}{4}$, $x^*_3$ is nonhyperbolic.
It can be shown that the quadratic function $\varphi(x) - \frac{2}{5}$ has roots at $\frac{2}{5}$ and $\frac{6}{5}$. Additionally,  $\varphi(x) - \frac{2}{5}$ is concave for all $x$. Therefore,
$$
\left\{\begin{array}{llll}
\varphi(x)-\frac{2}{5} > 0 \Rightarrow \varphi(x) > \frac{2}{5} & \hspace{3mm}, & x < \frac{2}{5}, \\[4mm]
\varphi(x)-\frac{2}{5} < 0 \Rightarrow \varphi(x) < \frac{2}{5} & \hspace{3mm}, & x > \frac{2}{5}.
\end{array}\right.
$$
which demonstrates that ${x_3^{(n)}}$ oscillates between $[0, \frac{2}{5})$ and $(\frac{2}{5},1]$. Moreover, it can be demonstrated that
$$
\left\{\begin{array}{llll}
x < \varphi(\varphi(x)) & \hspace{3mm}, & x < \frac{2}{5}, \\[4mm]
x > \varphi(\varphi(x)) & \hspace{3mm}, & x > \frac{2}{5}.
\end{array}\right.
$$
Thus, $\lim{x_3^{(n)}} = \frac{2}{5}$.

It can be seen from $|\varphi^{'}(x_3^*)| = |1-\sqrt{5-4e}|$ that when $e \in (\frac{1}{4}, 1)$, $x^*_3$ is an attracting fixed point of $\varphi(x)$.

Therefore, $x_3^{(n)}$ will converge to $x^*_3$ when $e \geq \frac{1}{4}$. On the invariant line $\gamma$ a trajectory of this operator is as follows:
$$
x^{(n)}=(x^{(n)}_1, 1-x^{(n)}_1-x^*_3,  x^*_3),
$$
where $x^{(n)}_1$ satisfies the equality
\begin{equation} \label{x1Trajectory}
x^{(n+1)}_1 = x^*_3(2-(2-c)x^*_3-2x^{(n)}_1).
\end{equation}
It follows from (\ref{x1Trajectory}) that $\lim_{n\rightarrow \infty} x^{(n)}_1 = x^*_1$. Therefore, when $\lim_{n\rightarrow \infty} x_3^{(n)} = x_3^*$, then $\lim_{n\rightarrow \infty}V(x^0) = (x_1^*,x_2^*,x_3^*)$.
\end{proof}

\section{Conclusion}
As was discussed in the introduction, there does not exist a general theory for non-Volterra quadratic stochastic operators. This paper represents an additional step towards a more comprehensive understanding of this family of operators. A complete understanding of non-Volterra QSOs would not only be a significant advance in the field of mathematical genetics and dynamical systems, but it would also answer questions about the modeling of populations that have complex genetic structures for certain traits. Further research in this area could include a more complete description of the $\w$-limit set of the operator studied here, as well as an investigation into a more general theory for non-Volterra QSOs.

\section*{ Acknowledgements}

The first author was supported by the National Science Foundation, grant number 1658672 

\vspace{\baselineskip}
\renewcommand{\abstractname}{Summary}

\end{document}